\documentclass{amsproc}
\usepackage{mathrsfs} %orignal

\usepackage{amsmath,amssymb}
\usepackage{graphics}
\newtheorem{theorem}{Theorem}[section]
\newtheorem{lemm}[theorem]{Lemma}

\newtheorem{theo}[theorem]{Theorem}
\theoremstyle{definition}
\newtheorem{defi}[theorem]{Definition}

\theoremstyle{remark}

\numberwithin{equation}{section}

\def\frak{\mathfrak}

\def\om{\omega}

\def\om{\omega}

\newfont{\df}{eufm10}

\def\ep{\epsilon}

\def\om{\omega}

\def\al{\alpha}
\begin{document}

\title[Fermionic
realization of $U_{r,s}({C}_n^{(1)})$]
    {Fermionic
realization of two-parameter quantum affine algebra
$U_{r,s}({C}_l^{(1)})$ }

%%% ----------------------------------------------------------------------
\author[Jing]{Naihuan Jing}
\address{School of Mathematical Sciences, South China University of Technology, Guangzhou 510640, China
and Department of Mathematics,
   North Carolina State University,
   Raleigh, NC 27695, USA}%\\
\email{jing@math.ncsu.edu}

\author[Zhang]{Honglian Zhang$^{\star}$}
\address{Department of Mathematics,
Shanghai University, Shanghai 200444, China}
\email{hlzhangmath@shu.edu.cn}
\thanks{$^\star$H. Zhang, Corresponding Author}

%    General info
%\date{Version on Jan. 2013}

%%%%%%%%%%%%%%%%%%%%%%%%%%%%%%%%%%%%%%%%%%%%%%%%%%%%%%%%%%%%%%%%%%%%%%%%
%\footnote{Corresponding author.}%
%%%%%%%%%%%%%%%%%%%%%%%%%%%%%%%%%%%%%%%%%%%%%%%%%%%%%%%%%%%%%%%%%%%%%%%%
\begin{abstract}  We construct a Fock space representation and
the action of  the two-parameter quantum algebra
$U_{r,s}(\frak{gl}_{\infty})$ using extended Young diagrams. In
particular, we obtain an integrable representation of the two-parameter
quantum affine algebra of type $C_n^{(1)}$ which is a two-parameter
generalization of Kang-Misra-Miwa's realization.
\end{abstract}

\maketitle

{\bf R\'esum\'e}. Nous construisons une repr\'esentation de l'espace de Fock et
l'action de la 2-param\`etre quantique alg\`ebre $U_{r,s}(\frak{gl}_{\infty})$
en  utilisant diagrammes de Young prolong\'ees. Dans
particulier, on obtient une repr\'esentation int\'egrable de la 2-param\`etres
quantique alg\`{e}bre affine pour le type $C_n^{(1)}$ qui est un %\`{a} deux
2-param\`{e}tres
g\'en\'eralisation de la r\'{e}alisation de Kang-Misra-Miwa.

\section{ Introduction}\

Quantum groups, introduced independently by Drinfeld \cite{Dr} and
Jimbo \cite{Jb}, are deformations of the universal enveloping algebras
of the Kac-Moody Lie algebras. Among the most important classes of quantum
groups, quantum affine algebras have a rich representation theory
and broad applications in mathematics and physics. In
particular they are expected to provide the mathematical foundation
for $q$-conformal field theory.

Two-parameter quantum groups associated to ${\frak{gl_n}}$ and
${\frak{sl_n}}$ were studied in \cite{BW1,BW2, BW3} by Benkart and
Witherspoon (see also earlier work by Takeuchi \cite{T}). Other
classical types and some exceptional types of two-parameter quantum
groups and their representations have been investigated in
\cite{BGH1, BGH2, HS} (see references therein). The two-parameter
quantum affine algebras were introduced in \cite{HRZ} and their
Drinfeld realization and vertex operator representations were also
known with help of Lyndon bases for type $A$. More recently
these structures have been generalized to all classical untwisted types in
\cite{HZ1, GHZ}, which are analog of the basic representations of
the quantum affine algebras \cite{FJ}. The latter builds upon
certain quantization of the so-called bosonic fields. From the other angle
aimed toward a categorification, \cite{JZ2} provided a
group-theoretic realization of two-parameter quantum toroidal
algebras using finite subgroups of $SL_2(\mathbb{C})$ via the McKay
correspondence.

It is well known that quantum affine algebras also admit fermionic
realizations \cite{H, MM, KMM, LT} that have played an important role in
integrable systems and representation theory. In \cite{JZ1} such a
fermionic realization of the two-parameter quantum affine algebra was
constructed for type $A$ using Young diagrams. The combinatorial
model gives rise to a natural interpretation of the deforming parameters $r$
and $s$. In this paper, we construct a fermionic
realization of the two-parameter quantum affine algebra of type
$C$ along the same line. We have taken a slightly different presentation from
\cite{JZ1} to use the approach of Kang-Misra-Miwa \cite{KMM}. We expect that
this model will also work for other 2-parameter twisted quantum affine algebras.

\section{The Fock Space of  $U_{r, s}(gl(\infty))$}
In this section, we first define the two-parameter quantum algebra
$U_{r, s}(gl(\infty))$, and obtain an
 irreducible integrable representation using extended Young diagrams.

Let $\{ \ep_i,
|i\in \mathbb{Z} \}$  be an orthonormal basis of a Euclidean space $E$
with an inner product $(\,,)$.
Let $\{\alpha_i|i \in \mathbb{Z}\}$ be the simple roots of the Lie algebra $\mathfrak{g}=gl(\infty)$.

We assume that the ground field $\mathbb{K}$ is the field
$\mathbb{Q}(r,s)$ of rational functions in $r, s$.
Similar to the definition of $U_{r, s}(gl_n)$ (cf. \cite{BW1}), we
define $U_{r, s}(gl(\infty))$ as follows.
\begin{defi} Let  $U_{r, s}(gl(\infty))$ be the unital associative algebra over $\mathbb{K}$ generated by
the elements $e^{\infty}_i, f^{\infty}_i, \omega^{\infty}_i,
{\omega'}^{\infty}_i$ for $i \in \mathbb{Z}$ satisfying the
following defining relations:
\begin{eqnarray*}
&(R1) \quad &  {(\omega^{\infty}_i)}^{\pm 1},
{({{\omega'}^{\infty}_j})}^{\pm 1}~~\hbox{all commute with each another and}\\
&&~~~~~~~~~~~~{\omega^{\infty}_i}{({\omega^{\infty}_i})}^{-1}={{\omega'}^{\infty}_j}{({{\omega'}^{\infty}_j})}^{-1}=1, \\
&(R2) \quad &  {\omega_i^{\infty}}
{e^{\infty}_j}=r^{(\varepsilon_{i},\,\alpha_j)}\, {e^{\infty}_j}
{\omega_i^{\infty}} \qquad ~~~~\hbox{and}~~~~ {\omega_i^{\infty}}
{f^{\infty}_j}=r^{-(\varepsilon_{i},\,\alpha_j)} \,{f^{\infty}_j}
{\omega_i^{\infty}},\\
&(R3) \quad &  {{\omega'}^{\infty}_i}
{e^{\infty}_j}=s^{(\varepsilon_{i},\,\alpha_j)}\, {e^{\infty}_j}
{{\omega'}^{\infty}_i} \qquad~~~~\hbox{and}~~~~
{{\omega'}^{\infty}_i}
{f^{\infty}_j}=s^{-(\varepsilon_{i},\,\alpha_j)}\, {f^{\infty}_j}
{{\omega'}^{\infty}_i},\\
&(R4) \quad & [\,{e^{\infty}_i},\,
{f^{\infty}_j}\,]=\frac{\delta_{ij}}{r-s}({\omega^{\infty}_i{\omega'}^{\infty}_{i+1}}-{\omega^{\infty}_{i+1}{\omega'}^{\infty}_i}),\\
&(R5) \quad &[\,{e^{\infty}_i},\,
{e^{\infty}_j}\,]=[\,{f^{\infty}_i},\,
{f^{\infty}_j}\,]=0~~~\hbox{if}~~|i-j|>1,\\
&(R6) \quad & {({e^{\infty}_i})}^2{e^{\infty}_{i+1}}-(r+s)\,{e^{\infty}_i}{e^{\infty}_{i+1}}{e^{\infty}_i}
+rs \,{e^{\infty}_{i+1}}{({e^{\infty}_{i}})}^2=0,\\
& &e^{\infty}_i{(e^{\infty}_{i+1})}^2-(r+s)\,e^{\infty}_{i+1}e^{\infty}_{i}e^{\infty}_{i+1}
+ rs \,{(e^{\infty}_{i+1})}^2e^{\infty}_{i}=0,\\
&(R7) \quad & {(f^{\infty}_i)}^2f^{\infty}_{i+1}-(r^{-1}+s^{-1})\,f^{\infty}_i f^{\infty}_{i+1} f^{\infty}_i
+{(rs)}^{-1}\, f^{\infty}_{i+1}{(f^{\infty}_{i})}^2=0,\\
&
&f^{\infty}_i{(f^{\infty}_{i+1})}^2-(r^{-1}+s^{-1})\,f^{\infty}_{i+1}f^{\infty}_{i}f^{\infty}_{i+1}+
{(rs)}^{-1} {(f^{\infty}_{i+1})}^2f^{\infty}_{i}=0.
\end{eqnarray*}
%where
% \begin{eqnarray*}\begin{array}{llll}
%&\langle i,\, j\rangle_{\infty}=\left\{\begin{array}{cl} rs^{-1},&  i=j ;\\
% r^{-1},& i=j-1 ;\\
% s, & i=j+1;\\
% 1, & otherwise.\\
%\end{array}\right.
%\end{array}
%\end{eqnarray*}
\end{defi}
%\begin{remark}\ (1) \ One can define the Hopf algebra structure for the two-parameter quantum algebra $U_{r, s}(gl(\infty))$ as usual, see \cite{BW1}, \cite{BGH1}
%for details. \\
%\ (2) \ Under the condition of $r=s^{-1}$, the Hopf algebra $U_{r, s}(gl(\infty))$
% is isomorphic to the usual one-parameter
% quantum  algebra $U_{q}(gl(\infty))$  modulo the ideal generated by
%the element $\{\,\om'^{\infty}_i-{({\om_i^{\infty}})}^{-1}|i\in \mathbb{Z}\}$.
%
%\end{remark}

Now we construct a Fock space representation for the two-parameter quantum algebra
$U_{r,s}(\frak{gl}_{\infty})$, which generalizes the fermionic representation
of the usual quantum algebra given in \cite{KMM}.

We begin with the definition of extended Young diagram given in \cite{JMMO}.
\begin{defi} An extended Young diagram $Y$ is a sequence $(y_k)_{k\geq
0}$ such that
\begin{eqnarray*} &&(i)\quad y_k\in \mathbb{Z},\quad y_k\leq y_{k+1} ~~\hbox{for~~ all} ~~k,\\
&&(ii)\quad \hbox{there~~ exists~~ fixed~~ integer}~~ y_{\infty}~~
\hbox{such~~ that}~~ y_k=y_\infty ~~\hbox{for}~~ k\gg 0.
\end{eqnarray*}
The integer $y_{\infty}$ is called the charge of $Y$.
\end{defi}

% For instance the extended Young diagram $\lambda=(-2, -2, -1, 0, 0, 1, 1, 1,\cdots )$ is
%given in Figure 1.
%
%\vspace{5cm}
%\begin{figure}[htb]
%
%\begin{center}
%
%\scalebox{0.5}{\includegraphics[300,-20]{2.jpg}}
%
%\caption{Extended Young diagram.}
%\end{center}
%\end{figure}

%\vspace{1cm}

 Another way to
identify an extended Young diagram is by specifying the fourth quadrant of the xy-plane with sites
$\{(i,j)\in \mathbb{Z}\times \mathbb{Z}|i\geq 0,\, j\leq 0\}$.
Thus an extended Young diagram $Y=(y_k)_{k\geq 0}$ is an infinite Young diagram
 drawn on the lattice in the right half plane with sites $\{(i,j)\in \mathbb{Z}\times \mathbb{Z}|i\geq 0,\, j\leq 0\}$, where $y_k$ denotes the ``depth'' of the $k$-th column.

Note that if $y_k\neq y_{k+1}$ for some $k$, then we will have
corners in the extended Young diagram $Y=(y_k)_{k\geq o}$.
%For instance, in the
%above example $y_1\neq y_2$, So we have a convex corner at site (2, -2) and a concave
%corner at (2, -1).
A corner is either ``concave'' or ``convex''. A corner located at site $(i,
j)$ is called a $d-$diagonal corner (or corner with diagonal number
d), where $d=i+j$. For more details please see \cite{JMMO} and \cite{KMM}.

For any fixed integer $n$, let $\phi_n$ denote the ``empty'' diagram
$(n, n, n, \cdots )$ of charge $n$. Let $Y_n$  denote the set of all
extended Young diagrams of charge $n$. The Fock space of charge $n$
$$\mathscr{F}_n=\bigoplus_{Y\in y_n}\mathbf{Q}(r, s)Y$$
denotes the $\mathbf{Q}(r, s)$-vector space having all $Y\in Y_n$ as
base vectors.

The algebra $U_{r, s}(gl(\infty))$ acts on the Fock space as
follows:
\begin{theo} $\mathscr{F}_n$
is an irreducible integrable $U_{r, s}(gl(\infty))$- module under the
action defined as follows. For $Y\in y_n$,
\begin{eqnarray*}
{e^{\infty}_i}Y&=Y',  \qquad & \hbox{if Y has an i-diagonal convex corner then},\\
&& \hbox{$Y'$ is the same as Y except that the $i$-diagonal }\\
&& \hbox{convex corner is replaced by a concave corner,}\\
&=0, \qquad & otherwise;\\
{f^{\infty}_i}Y&=Y'',  \qquad & \hbox{if Y has an i-diagonal concave corner then},\\
&& \hbox{$Y''$ is the same as Y except that the $i$-diagonal }\\
&& \hbox{concave corner is replaced by a convex corner,}\\
&=0, \qquad & otherwise;\\
{\omega^{\infty}_i}Y &=s^{-1}Y,  \qquad & \hbox{if Y has an i-diagonal concave corner },\\
&=r^{-1}, \qquad & \hbox{if $Y$ has an i-diagonal convex corner },\\
&=Y, \qquad & otherwise;\\
{{\omega'}^{\infty}_i}Y&=rY,  \qquad & \hbox{if Y has an i-diagonal concave corner },\\
&=s, \qquad & \hbox{if $Y$ has an i-diagonal convex corner },\\
&=Y, \qquad & otherwise.
\end{eqnarray*}
\end{theo}
\begin{proof} It is straightforward to verify the relations $(R1)-(R7)$ for the action
on $\mathscr{F}_n$ for all generators. We remark that this is very much the same as
in type $A$ situation \cite{JZ1}.

\end{proof}

\section{Fock Space Representations of  $U_{r,s}({C}_l^{(1)})$}
\medskip

Having constructed the Fock space representation of the two-parameter quantum affine algebra
$U_{r, s}(gl(\infty))$, we can build the Fock space representation of $U_{r,s}({C}_l^{(1)})$
by generalizing the well--known embedding of the latter inside $U_{r, s}(gl(\infty))$.  First let us recall the definition of the two-parameter
quantum affine algebra $U_{r,s}({C}_l^{(1)})$ from \cite{GHZ}.

 Let $I_0=\{0,\,1,\,2,\,\cdots, n\}$, and $(a_{ij}),\, i,\,j\in I_0$ be the Cartan matrix  of type $C_l^{(1)}$.
 %the finite dimensional simple Lie algebra $\mathfrak{g}(C_l^{(1)})$.
 We take the normalization $(\alpha_0,\, \alpha_0)=(\alpha_l,\,
\alpha_l)=1$ and $(\alpha_i,\, \alpha_i)=\frac{1}{2}$ for $1\leq i\leq l-1$.
%$i=1,\,2,\cdots, l-1$.
Let $r_i=r^{\frac{(\al_i, \al_i)}{2}}$ and
$s_i=s^{\frac{(\al_i, \al_i)}{2}}$.
%~then
% $r_0=r_l=r,\,~r_1=\cdots=r_{l-1}=r^{\frac{1}{2}}$ and $s_0=s_l=s,\,
%~s_1=\cdots=s_{l-1}=s^{\frac{1}{2}}$.
Denote by $c$ the canonical central element of $\mathfrak{g}(C_l^{(1)})$ and let $\delta_{ij}$ denote the Kronecker symbol.

\begin{defi} {\it The {\em two-parameter quantum affine algebra
 $U_{r,s}(\mathrm{C}_n^{(1)})$} is the unital associative
algebra over $\mathbb{K}$ generated by the elements $e_j,\, f_j,\,
\omega_j^{\pm 1},\, \omega_j'^{\,\pm 1}\, (j\in I_0),\,
\gamma^{\pm\frac{1}2},\,\gamma'^{\pm\frac{1}2},$ $D^{\pm1},
D'^{\,\pm1}$, satisfying the following relations:

\noindent $(\hat{C}1)$
$\gamma^{\pm\frac{1}2},\,\gamma'^{\pm\frac{1}2}$ are central with
$\gamma=\om_\delta$, $\gamma'=\om'_\delta$, $\gamma\gamma'=(rs)^c$, such
that $\omega_i\,\omega_i^{-1}=\omega_i'\,\omega_i'^{\,-1}=1
=DD^{-1}=D'D'^{-1}$, and
\begin{equation*}
\begin{split}[\,\omega_i^{\pm 1},\omega_j^{\,\pm 1}\,]&=[\,\omega_i^{\pm1},
D^{\pm1}\,]=[\,\omega_j'^{\,\pm1}, D^{\pm1}\,] =[\,\omega_i^{\pm1},
D'^{\pm1}\,]=0\\
&=[\,\omega_i^{\pm 1},\omega_j'^{\,\pm 1}\,]=[\,\om_j'^{\,\pm1},
D'^{\pm1}\,]=[D'^{\,\pm1}, D^{\pm1}]=[\,\omega_i'^{\pm
1},\omega_j'^{\,\pm 1}\,].
\end{split}
\end{equation*}
$(\hat{C}2)$ \ \textit{For} $\,0 \leqslant i,\, j \leqslant l$,
\begin{equation*}\begin{array}{lll}
&D\,e_i\,D^{-1}=r_i^{\delta_{0i}}\,e_i,\qquad\qquad\quad
&D\,f_i\,D^{-1}=r_i^{-\delta_{0i}}\,f_i,\\
&\omega_je_i\omega_j^{-1}=\langle i,\, j \rangle\,e_i,~~~~~~~~~~~
&\omega_jf_i\omega_j^{-1}=\langle j,\, i \rangle^{-1}\,f_i.
\end{array}\end{equation*}
$(\hat{C}3)$ \ \textit{For} $\,0 \leqslant i,\, j \leqslant l$,
\begin{equation*}\begin{array}{lll}
 &D'\,e_i\,D'^{-1}=s_i^{\delta_{0i}}\,e_i,\qquad\qquad\quad
&D'\,f_i\,D'^{-1}=s_i^{-\delta_{0i}}\,f_i,\\
&\omega'_je_i{\omega'}_j^{-1}=\langle i,\, j
\rangle^{-1}\,e_i,~~~~~~~~~~~ &\omega'_jf_i{\omega'}_j^{-1}=\langle
j,\, i \rangle\,f_i.
\end{array}\end{equation*}
$(\hat{C}4)$ \ \textit{For} $\,0 \leqslant i,\, j \leqslant l$,
 $$[\,e_i, f_j\,]=\frac{\delta_{ij}}{r_i-s_i}(\omega_i-\omega'_i).$$
$(\hat{C}5)$
  \textit{For all}  $1\leqslant i\ne j\leqslant l$ \textit{but}
  $(i, j)\not\in\{(0, l), (l, 0)\}$ \textit{such that} $a_{ij}=0$,
 $$[e_i, e_j]=[f_i, f_j]=0,$$
 $$e_le_0=rs\,e_0e_l,\qquad f_0f_l=rs\,f_lf_0.$$
$(\hat{C}6)$ \ \textit{For} $1\leqslant i\leqslant l-2$,
\textit{~the }$(r,s)$-Serre \textit{relations for $e_i's$:}
\begin{gather*}
e_0^2e_{1}-(r{+}s)\,e_0e_{1}e_0+rs\,e_{1}e_0^2=0,\\
e_i^2e_{i+1}-(r_i{+}s_i)\,e_ie_{i+1}e_i+(r_is_i)\,e_{i+1}e_i^2=0,\\
e_{i+1}^2e_i-(r_{i+1}^{-1}{+}s_{i+1}^{-1})\,e_{i+1}e_ie_{i+1}+
(r_{i+1}^{-1}s_{i+1}^{-1})\,e_ie_{i+1}^2=0,\\
e_l^2e_{l-1}-(r^{-1}{+}s^{-1})\,e_le_{l-1}e_l+(r^{-1}s^{-1})\,e_{l-1}e_l^2=0,\\
e_{l-1}^3e_l-(r{+}(rs)^{\frac{1}{2}}{+}s)\,e_{l-1}^2e_le_{l-1}\\
+(rs)^{\frac{1}{2}}\,
(r{+}(rs)^{\frac{1}{2}}{+}s)\,e_{l-1}e_le_{l-1}^2-
(rs)^{\frac{3}{2}}\,e_le_{l-1}^3=0\\
e_{1}^3e_0-(r^{-1}{+}(rs)^{-\frac{1}{2}}{+}s^{-1})\,e_{1}^2e_0e_{1}\\
+(rs)^{-\frac{1}{2}}\,
(r^{-1}{+}(rs)^{-\frac{1}{2}}{+}s^{-1})\,e_{1}e_0e_{1}^2-
(rs)^{-\frac{3}{2}}\,e_0e_{1}^3=0.
\end{gather*}
 $(\hat{C}7)$ \ \textit{For $1\leqslant i\leqslant l-2$,
~the $(r,s)$-Serre relations for $f_i's$  are obtained from $(\hat{C}6)$ by
replacing $e_i$ for $f_i$ and $r, s$ by $r^{-1}, s^{-1}$ respectively.}

% \textit{For} $1\leqslant i\leqslant l-2$, \textit{~there are the following }
%$(r,s)$-Serre \textit{relations:}
%\begin{gather*}
%f_{i+1}f_i^2-(r_i{+}s_i)\,f_if_{i+1}f_i+(r_is_i)\,f_i^2f_{i+1}=0,\\
%f_{1}f_0^2-(r{+}s)\,f_0f_{1}f_0+rs\,f_0^2f_{1}=0,\\
%f_if_{i+1}^2-(r_{i+1}^{-1}{+}s_{i+1}^{-1})\,f_{i+1}f_if_{i+1}+
%(r_{i+1}^{-1}s_{i+1}^{-1})\,f_{i+1}^2f_i=0,\\
%f_{l-1}f_l^2-(r^{-1}{+}s^{-1})\,f_lf_{l-1}f_l+(r^{-1}s^{-1})\,f_l^2f_{l-1}=0,\\
%f_lf_{l-1}^3-(r{+}(rs)^{\frac{1}{2}}{+}s)\,f_{l-1}f_lf_{l-1}^2\\
%+(rs)^{\frac{1}{2}}\,(r{+}(rs)^{\frac{1}{2}}{+}s)\,f_{l-1}^2f_lf_{l-1}-
%(rs)^{\frac{3}{2}}\,f_{l-1}^3f_l=0,\\
%f_0f_{1}^3-(r^{-1}{+}(rs)^{-\frac{1}{2}}{+}s^{-1})\,f_{1}f_0f_{1}^2\\
%+(rs)^{-\frac{1}{2}}\,
%(r^{-1}{+}(rs)^{-\frac{1}{2}}{+}s^{-1})\,f_{1}^2f_0f_{1}-
%(rs)^{-\frac{3}{2}}\,f_{1}^3f_0=0.
%\end{gather*}}

In the above $\langle i,\, j\rangle$ are the matrix entries of the
two-parameter quantum Cartan matrix
 for type $\mathrm{C}_{l}^{(1)}$:
$$\left(\begin{array}{cccccc}
rs^{-1}& r^{-1}& 1 & \cdots & 1 & rs \\
s & r^{\frac{1}{2}}s^{-\frac{1}{2}} & r^{-\frac{1}{2}}  & \cdots & 1 & 1\\
\cdots &\cdots &\cdots & \cdots & \cdots & \cdots\\
1 & 1 & 1  & \cdots & r^{\frac{1}{2}}s^{-\frac{1}{2}} & r^{-1}\\
 (rs)^{-1} & 1 & 1 & \cdots & s & rs^{-1}
\end{array}\right)$$}
\end{defi}

 We now desribe the integrable representation of the two-parameter quantum affine
 algebra $U_{r,s}({C}_l^{(1)})$. We start with the folding map
$$\pi: \qquad \{0,\,1,\cdots, 2l-1 \}\rightarrow  \{0,\,1,\cdots, l
\}$$ where $\pi(0)=0, \pi(l)=l$ and $\pi(i)=\pi(2l-i)=i$ for $1 \leq i
\leq l-1$. Extend $\pi$ to a map from $\mathbb{Z}$ into $\{0,\,1,\cdots, l \}$
by periodicity $2l$.

For any $Y=(y_k)_{k\geq 0}\in Y_n$ define the
operators:
$$t_kY=r^aY, \qquad t'_k=s^a Y$$
where $a=|\{p\in\mathbb{Z}|y_k<p\leq n \quad \hbox{and}\quad
\pi(p+k)=0\}|$ which depends on k.

As we still act on the space $\mathscr{F}_n$, so we continue to use the same notation for the new Fock space representation. The following theorem is proved by direct verification (see \cite{JZ1}).

\begin{theo}
For $k=0,\,1,\,\cdots, l$, the algebra $U_{r,s}(\mathrm{C}_l^{(1)})$
acts on $\mathscr{F}_n$ by the equations:
\begin{eqnarray}
&e_i= \sum\limits_{\substack{j\\ \pi(j)=i}} \Big( \prod\limits_{\substack{k>j\\
\pi(k)=i}}{\omega}_{k}^{\infty}\Big)^{(\alpha_i,\,\alpha_i)}e_j^{\infty},\\
&f_i=\sum\limits_{\substack{j\\ \pi(j)=i}}f_j^{\infty}\Big(
\prod\limits_{\substack{k<j\\
\pi(k)=i}}{{{\omega'}_k^{\infty}}}\Big)^{(\alpha_i,\,\alpha_i)},\\
&\omega_i=\Big( \prod\limits_{\substack{j\\
\pi(j)=i}}\omega_{j}^{\infty}\Big)^{(\alpha_i,\,\alpha_i)},\\
&\omega'_i=\Big( \prod\limits_{\substack{j\\
\pi(j)=i}}{\omega'}_{j}^{\infty}\Big)^{(\alpha_i,\,\alpha_i)},\\
&D=\prod\limits_{k\geq 0}t_k, \qquad D'=\prod\limits_{k\geq 0}t'_k.
\end{eqnarray}
Under the above action $\mathscr{F}_n$ is an integrable
$U_{r,s}(\mathrm{C}_l^{(1)})$-module.
\end{theo}

\begin{proof}\, We proceed in the same way.  First we have
$$\om'_j\,e_i\,{\om'}_j^{-1}=\Big( \prod\limits_{\substack{k\\
\pi(k)=j}}{\omega'}_{k}^{\infty}\Big)^{(\alpha_j,\,\alpha_j)}\,
 \sum\limits_{\substack{j\\ \pi(j)=i}} \Big( \prod\limits_{\substack{j'>j\\
\pi(j')=i}}{\omega}_{j'}^{\infty}\Big)^{(\alpha_i,\,\alpha_i)}e_j^{\infty}
\,\Big( \prod\limits_{\substack{k\\
\pi(k)=j}}{\omega'}_{k}^{\infty}\Big)^{-(\alpha_j,\,\alpha_j)}
$$

 We don't have to prove anything for $|i-j|\geqslant 2$ due to ${{\omega'}_i^{\infty}}
{e^{\infty}_j}={e^{\infty}_j} {{\omega'}_i^{\infty}} $.
 For $i=j$, we have
$e^{\infty}_m\,({{\omega'}_m^{\infty}})^{-1}=r^{-1}s\,({{\omega'}_m^{\infty}})^{-1}\,e^{\infty}_m$,
which follows from $\langle i,\,
i\rangle^{-1}=r^{-(\alpha_i,\,\alpha_i)}s^{(\alpha_i,\,\alpha_i)}$.
For $0 \leqslant i=j-1 \leqslant l-1$, applying
$e^{\infty}_m\,({\om'}^{\infty}_{m+1})^{-1}=s^{-1}\,({\om'}^{\infty}_{m+1})^{-1}\,e^{\infty}_m$
and $\langle i+1,\,
i\rangle^{-1}=s^{-(\alpha_{i+1},\,\alpha_{i+1})}$, we arrive at the required relation.
Finally, when $1 \leqslant i=j+1 \leqslant l $, we have
$e^{\infty}_m\,({\om'}^{\infty}_{m-1})^{-1}=r\,({\om'}^{\infty}_{m-1})^{-1}\,e^{\infty}_m$
and $\langle i-1,\,
i\rangle^{-1}=r^{(\alpha_{i-1},\,\alpha_{i-1})}$, and this  implies the conclusion.
%\end{proof}

For further reference, we need a few useful identities.

\begin{lemm} By direct calculations, we get the actions on $\mathscr{F}_n$,
\begin{eqnarray*}
&f_m^{\infty}\,({\om_{m'}^{\infty}})^{-1}=\langle m,
m'\rangle_{\infty}^{-1}({\om_{m'}^{\infty}})^{-1}\,f_m^{\infty},\\
&e_k^{\infty}\,({\om_{m'}^{\infty}})^{-1}=\langle k,
m'\rangle_{\infty}({\om_{m'}^{\infty}})^{-1}\,e_k^{\infty},\\
&f_m^{\infty}\,({{\om'}_{k'}^{\infty}})^{-1}=\langle m,
k'\rangle_{\infty}({{\om'}_{k'}^{\infty}})^{-1}\,f_m^{\infty}.
\end{eqnarray*}
where $\langle i,\,j\rangle_{\infty}$ is defined as follows:
 \begin{eqnarray*}\begin{array}{llll}
&\langle i,\, j\rangle_{\infty}=\left\{\begin{array}{cl} rs^{-1},&  i=j ;\\
 r^{-1},& i=j-1 ;\\
 s, & i=j+1;\\
 1, & otherwise.\\
\end{array}\right.
\end{array}
\end{eqnarray*}

\end{lemm}

Now we turn to the relation $(\hat{C}4)$. From definition and Lemma 3.3, it follows that
\begin{eqnarray*}
&&e_if_j-f_je_i\\
&=& \sum\limits_{\substack{k\\ \pi(k)=i}} \Big( \prod\limits_{\substack{k'>k\\
\pi(k')=i}}{\omega}_{k'}^{\infty}\Big)^{(\alpha_i,\,\alpha_i)}e_k^{\infty}\,\sum\limits_{\substack{m\\
\pi(m)=j}}f_m^{\infty}\Big(
\prod\limits_{\substack{m'<m\\
\pi(m')=j}}{{{\omega'}_{m'}^{\infty}}}\Big)^{(\alpha_j,\,\alpha_j)}\\
&&-\sum\limits_{\substack{m\\
\pi(m)=j}}f_m^{\infty}\Big(
\prod\limits_{\substack{m'<m\\
\pi(m')=j}}{{{\omega'}_{m'}^{\infty}}}\Big)^{(\alpha_j,\,\alpha_j)}\,
\sum\limits_{\substack{k\\ \pi(k)=i}} \Big( \prod\limits_{\substack{k'>k\\
\pi(k')=i}}{\omega}_{k'}^{\infty}\Big)^{(\alpha_i,\,\alpha_i)}e_k^{\infty}\\
&=&\sum\limits_{\substack{k, m\\ \pi(k)=i\\ \pi(m)=j}}\Big[ \Big(\prod\limits_{\substack{k'>k\\
\pi(k')=i}}{\omega}_{k'}^{\infty}\Big)^{(\alpha_i,\,\alpha_i)}e_k^{\infty}\,
 f_m^{\infty}\Big(
\prod\limits_{\substack{m'<m\\
\pi(m')=j}}{{{\omega'}_{m'}^{\infty}}}\Big)^{(\alpha_j,\,\alpha_j)}\\
&&- f_m^{\infty}\Big(
\prod\limits_{\substack{m'<m\\
\pi(m')=j}}{{{\omega'}_{m'}^{\infty}}}\Big)^{(\alpha_j,\,\alpha_j)}\,
\Big(\prod\limits_{\substack{k'>k\\
\pi(k')=i}}{\omega}_{k'}^{\infty}\Big)^{(\alpha_i,\,\alpha_i)}e_k^{\infty}\,
\Big]
\end{eqnarray*}
\begin{eqnarray*}
&=&\sum\limits_{\substack{k>m\\ \pi(k)=i\\ \pi(m)=j}} \Big(\prod\limits_{\substack{k'>k\\
\pi(k')=i}}{\omega}_{k'}^{\infty}\Big)^{(\alpha_i,\,\alpha_i)}\Big(
\prod\limits_{\substack{m'<m\\
\pi(m')=j}}{{{\omega'}_{m'}^{\infty}}}\Big)^{(\alpha_j,\,\alpha_j)}\,(e_k^{\infty}\,
 f_m^{\infty}- f_m^{\infty}\,e_k^{\infty})\\
&&+\delta_{i,j}\sum\limits_{\substack{k\\ \pi(k)=i}}\Big(\prod\limits_{\substack{k'>k\\
\pi(k')=i}}{\omega}_{k'}^{\infty}\Big)^{(\alpha_i,\,\alpha_i)}\Big(
\prod\limits_{\substack{k'<k\\
\pi(k')=j}}{{{\omega'}_{k'}^{\infty}}}\Big)^{(\alpha_j,\,\alpha_j)}\,(e_k^{\infty}\,
 f_k^{\infty}- f_k^{\infty}\,e_k^{\infty})\\
 &&+\sum\limits_{\substack{k<m\\ \pi(k)=i\\ \pi(m)=j}} \Big(\prod\limits_{\substack{k'>k\\
\pi(k')=i}}{\omega}_{k'}^{\infty}\Big)^{(\alpha_i,\,\alpha_i)}\Big(
\prod\limits_{\substack{m'<m\\
\pi(m')=j}}{{{\omega'}_{m'}^{\infty}}}\Big)^{(\alpha_j,\,\alpha_j)}\\
&&\hskip2cm\Big(\sum\limits_{m'<m}\langle k,
m'\rangle_{\infty}^{(\alpha_j,\,\alpha_j)}e_k^{\infty}\,
 f_m^{\infty}- \sum\limits_{k'>k}\langle
m,
k'\rangle_{\infty}^{(\alpha_i,\,\alpha_i)}f_m^{\infty}\,e_k^{\infty}\Big)
\end{eqnarray*}

Note that if $m=k+1$, then we have
$e_k^{\infty}\,f_m^{\infty}=0=f_m^{\infty}\,e_k^{\infty}$, and if
$m>k+1$, then
\begin{gather}
\sum\limits_{\substack{m'<m\\ \pi(k)=i\\ \pi(m)=j=\pi(m')}}\langle
k,
m'\rangle_{\infty}^{(\alpha_j,\,\alpha_j)}=\sum\limits_{\substack{k'>k\\ \pi(m)=j\\
\pi(k)=i=\pi(k')}}\langle m,
k'\rangle_{\infty}^{(\alpha_i,\,\alpha_i)}
\end{gather}

On $\mathscr{F}_n$, it is clear that
\begin{gather}
e_k^{\infty}\,
 f_m^{\infty}- f_m^{\infty}\,e_k^{\infty}=\delta_{k,m}
 \Big(\frac{(\om^{\infty}_k)^{(\alpha_i,\,\alpha_i)}-{\om'}^{\infty}_k)^{(\alpha_i,\,\alpha_i)}}{r_i-s_i} \Big)
\end{gather}
Consequently,  it follows that on $\mathscr{F}_n$,
\begin{eqnarray*}
&&e_if_j-f_je_i\\
&=&\delta_{i,j}(r_i-s_i)^{-1}\sum\limits_{\substack{k\\
\pi(k)=i}}\{\Big(\prod\limits_{k'\geqslant
k}{\omega}_{k'}^{\infty}\Big)^{(\alpha_i,\,\alpha_i)}\Big(
\prod\limits_{k'<k}{{{\omega'}_{k'}^{\infty}}}\Big)^{(\alpha_i,\,\alpha_i)}\\
&&-\Big(\prod\limits_{k'>
k}{\omega}_{k'}^{\infty}\Big)^{(\alpha_i,\,\alpha_i)}\Big(
\prod\limits_{k'\leqslant
k}{{{\omega'}_{k'}^{\infty}}}\Big)^{(\alpha_i,\,\alpha_i)}\}\\
&=&\delta_{i,j}(r_i-s_i)^{-1}\{
\Big(\prod\limits_{\substack{k\\
\pi(k)=i}}{\omega}_{k}^{\infty}\Big)^{(\alpha_i,\,\alpha_i)}-\Big(\prod\limits_{\substack{k\\
\pi(k)=i}}{\omega'}_{k}^{\infty}\Big)^{(\alpha_i,\,\alpha_i)}\}\\
&=&\delta_{i,j}\,\frac{\om_i-{\om'}_i}{r_i-s_i}.
 \end{eqnarray*}

It is straightforward to check the relation $(\hat{C}5)$,
\begin{eqnarray*}
e_le_0 &=&\sum\limits_{\substack{k\\ \pi(k)=l}} \Big( \prod\limits_{\substack{k'>k\\
\pi(k')=l}}{\omega}_{k'}^{\infty}\Big)^{(\alpha_l,\,\alpha_l)}e_k^{\infty}\,
\sum\limits_{\substack{m\\ \pi(m)=0}} \Big( \prod\limits_{\substack{m'>m\\
\pi(m')=0}}{\omega}_{m'}^{\infty}\Big)^{(\alpha_0,\,\alpha_0)}e_m^{\infty}\\
&=&r\,\sum\limits_{\substack{k, m\\ \pi(k)=l\\ \pi(m)=0}} \Big( \prod\limits_{\substack{k'>k\\
\pi(k')=l}}{\omega}_{k'}^{\infty}\Big)^{(\alpha_l,\,\alpha_l)}\,
 \Big( \prod\limits_{\substack{m'>m\\
\pi(m')=0}}{\omega}_{m'}^{\infty}\Big)^{(\alpha_0,\,\alpha_0)}e_k^{\infty}e_m^{\infty}\\
&=&rs\,\sum\limits_{\substack{k, m\\ \pi(k)=l\\ \pi(m)=0}} \Big( \prod\limits_{\substack{m'>m\\
\pi(m')=0}}{\omega}_{m'}^{\infty}\Big)^{(\alpha_0,\,\alpha_0)}e_m^{\infty}\,\Big( \prod\limits_{\substack{k'>k\\
\pi(k')=l}}{\omega}_{k'}^{\infty}\Big)^{(\alpha_l,\,\alpha_l)}
 e_k^{\infty}\\
 &=&rs\, e_0e_l.
 \end{eqnarray*}
The others relations can be proved  similarly.

The last task is to verify the Serre relations $(\hat{C}6)$ and $(\hat{C}7)$. Here we
only check the relation $(\hat{C}6)$ as the other relations are proved exactly in the same way.

To show the Serre relations $(\hat{C}6)$, let us begin with the following notation to save space.
\begin{gather}
p_j=\prod\limits_{\substack{j'>j\\
\pi(j')=\pi(j)}}{\omega}_{j'}^{\infty}\\
p'_j=\prod\limits_{\substack{j'>j\\
\pi(j')=\pi(j)}}{\omega'}_{j'}^{\infty}.
\end{gather}

Let us write $i \gg j$ if $i-j>2$. The following lemmas can be checked directly.

\begin{lemm}
For all $j$ and $k$, on $\mathscr{F}_n$ then we obtain,
\begin{gather}
 e_k^{\infty}\, e_k^{\infty}=0, \vspace*{3mm}\\
  e_k^{\infty}\, e_j^{\infty}\, e_k^{\infty}=0.
\end{gather}
\end{lemm}

\begin{lemm}
If $\pi(k)=0=\pi(j)$, then it holds
\begin{gather}
e_j^{\infty}\,p_k=\left\{\begin{array}{cl}p_k\,e_j^{\infty},
& \textit{for}\,\, j\leqslant k;\vspace*{2mm}\\
r^{-1}s\,p_k\,e_j^{\infty}, &\textit{for}\,\, j> k.
\end{array}\right.\\
e_j^{\infty}\,p'_k=\left\{\begin{array}{cl}p'_k\,e_j^{\infty}, &
\textit{for}\,\, j\leqslant k; \vspace*{2mm}\\
 rs^{-1}\,p'_k\,e_j^{\infty}, &\textit{for}\,\, j> k.
\end{array}\right.
\end{gather}
\end{lemm}

\begin{lemm}
If $\pi(j)=0, \pi(k)=1$, then it follows that
\begin{gather}
e_j^{\infty}\,p_k=\left\{\begin{array}{cl}p_k\,e_j^{\infty},
& \textit{for}\,\, j\leqslant k;\vspace*{2mm}\\
r\,p_k\,e_j^{\infty}, &\textit{for}\,\, j=k+1;\vspace*{2mm}\\
r^{2}\,p_k\,e_j^{\infty}, &\textit{for}\,\, j\gg k.
\end{array}\right.\\
e_j^{\infty}\,p'_k=\left\{\begin{array}{cl}p'_k\,e_j^{\infty},
& \textit{for}\,\, j\leqslant k;\vspace*{2mm}\\
s\,p'_k\,e_j^{\infty}, &\textit{for}\,\, j=k+1;\vspace*{2mm}\\
s^{2}\,p'_k\,e_j^{\infty}, &\textit{for}\,\, j \gg k.
\end{array}\right.
\end{gather}
\end{lemm}

\begin{lemm}
If $\pi(j)=1, \pi(k)=0$, then we have
\begin{gather}
e_j^{\infty}\,p_k=\left\{\begin{array}{cl}p_k\,e_j^{\infty},
& \textit{for}\,\, j\leqslant k \,\textit{or}\, j=k+1;\vspace*{2mm}\\
s^{-1}\,p_k\,e_j^{\infty}, &\textit{for}\,\, j \gg k.
\end{array}\right.\\
e_j^{\infty}\,p'_k=\left\{\begin{array}{cl}p'_k\,e_j^{\infty},
& \textit{for}\,\, j\leqslant k \,\textit{or}\, j=k+1;\vspace*{2mm}\\
r^{-1}\,p'_k\,e_j^{\infty}, &\textit{for}\,\, j \gg k.
\end{array}\right.
\end{gather}
\end{lemm}

\begin{lemm}
If $\pi(j)=1=\pi(k)$, it is easy to see that
\begin{gather}
e_j^{\infty}\,p_k=\left\{\begin{array}{cl}p_k\,e_j^{\infty},
& \textit{for}\,\, j\leqslant k ;\vspace*{2mm}\\
r^{-1}s\,p_k\,e_j^{\infty}, &\textit{for}\,\, j>k.
\end{array}\right.\\
e_j^{\infty}\,p'_k=\left\{\begin{array}{cl}p'_k\,e_j^{\infty},
& \textit{for}\,\, j\leqslant k ;\vspace*{2mm}\\
rs^{-1}\,p'_k\,e_j^{\infty}, &\textit{for}\,\, j>k.
\end{array}\right.
\end{gather}
\end{lemm}

We now prove the following Serre relation:
\begin{eqnarray}
e_0^{2}e_1+(r+s)e_0e_1e_0+rse_1e_0^2=0
 \end{eqnarray}
We first use definition to simply the left hand side (LHS) of $(3.20)$.
\begin{eqnarray*}
&&LHS\\&=&\sum\limits_{\substack{j, k, m\\ \pi(j)=0=\pi(k)\\
\pi(m)=1}}\Big[ \Big( \prod\limits_{\substack{j'>j\nonumber\\
\pi(j')=0}}{\omega}_{k'}^{\infty}\Big)e_j^{\infty}\,\Big( \prod\limits_{\substack{k'>k\\
\pi(k')=0}}{\omega}_{k'}^{\infty}\Big)e_k^{\infty}\, \Big( \prod\limits_{\substack{m'>m\\
\pi(m')=1}}{\omega}_{m'}^{\infty}\Big)^{\frac{1}{2}}e_m^{\infty}\nonumber \\
&&-(r+s)\,\Big( \prod\limits_{\substack{j'>j\\
\pi(j')=0}}{\omega}_{k'}^{\infty}\Big)e_j^{\infty}\,\Big( \prod\limits_{\substack{m'>m\\
\pi(m')=1}}{\omega}_{m'}^{\infty}\Big)^{\frac{1}{2}}e_m^{\infty}\,\Big( \prod\limits_{\substack{k'>k \\
\pi(k')=0}}{\omega}_{k'}^{\infty}\Big)e_k^{\infty}\nonumber \\
&&+\Big( \prod\limits_{\substack{m'>m\\
\pi(m')=1}}{\omega}_{m'}^{\infty}\Big)^{\frac{1}{2}}e_m^{\infty}\,\Big( \prod\limits_{\substack{j'>j\\
\pi(j')=0}}{\omega}_{k'}^{\infty}\Big)e_j^{\infty}\,\Big( \prod\limits_{\substack{k'>k\\
\pi(k')=0}}{\omega}_{k'}^{\infty}\Big)e_k^{\infty}\Big]\nonumber \\
&=&\Big(\sum\limits_{m\gg j>k}+\sum\limits_{m=j+1\gg
k}+\sum\limits_{j\gg m\gg k}+\sum\limits_{j\gg
m=k+1}+\sum\limits_{j=m+1>k}+\sum\limits_{j>k\gg
m}+\sum\limits_{j>k=m+1}\Big)\nonumber \\
&&\times
\{p_je_j^{\infty}\,p_ke_k^{\infty}\,p_m^{\frac{1}{2}}e_m^{\infty}+p_ke_k^{\infty}\,p_je_j^{\infty}\,p_m^{\frac{1}{2}}e_m^{\infty}\nonumber\\
&&-(r+s)(p_je_j^{\infty}\,p_m^{\frac{1}{2}}e_m^{\infty}\,p_ke_k^{\infty}+p_ke_k^{\infty}\,p_m^{\frac{1}{2}}e_m^{\infty}\,p_me_m^{\infty})\nonumber\\
&&+(rs)(p_m^{\frac{1}{2}}e_m^{\infty}\,p_je_j^{\infty}\,p_ke_k^{\infty}+p_m^{\frac{1}{2}}e_m^{\infty}\,p_ke_k^{\infty}\,p_je_j^{\infty})\}.\nonumber
 \end{eqnarray*}
Using Lemma 3.4 through Lemma 3.8, we would like to show that each summand is actually 0. Taking the second summand for
example, we get immediately,
\begin{eqnarray*}
&&\sum\limits_{m=j+1\gg k}
\{p_je_j^{\infty}\,p_ke_k^{\infty}\,p_m^{\frac{1}{2}}e_m^{\infty}+p_ke_k^{\infty}\,p_je_j^{\infty}\,p_m^{\frac{1}{2}}e_m^{\infty}\\
&&-(r+s)(p_je_j^{\infty}\,p_m^{\frac{1}{2}}e_m^{\infty}\,p_ke_k^{\infty}+p_ke_k^{\infty}\,p_m^{\frac{1}{2}}e_m^{\infty}\,p_je_j^{\infty})\\
&&+(rs)(p_m^{\frac{1}{2}}e_m^{\infty}\,p_je_j^{\infty}\,p_ke_k^{\infty}+p_m^{\frac{1}{2}}e_m^{\infty}\,p_ke_k^{\infty}\,p_je_j^{\infty})\}\\
&=&\sum\limits_{m=j+1>k}\{(r^{-1}s+1-r^{-1}(r+s))p_jp_kp_m^{\frac{1}{2}}e_j^{\infty}e_m^{\infty}e_k^{\infty}\\
&&\hskip1cm+(-(r+s)s^{-1}+(rs)(r^{-1}s^{-1}+s^{-2}))
p_jp_kp_m^{\frac{1}{2}}e_m^{\infty}e_j^{\infty}e_k^{\infty} \}\\
&=&0.
 \end{eqnarray*}
The other summands are seen as zero by the same method. Subsequently Relation $(3.20)$ has been verified.

Next we turn to the relation
\begin{eqnarray}
e_{1}^3e_0-(r^{-1}{+}(rs)^{-\frac{1}{2}}{+}s^{-1})\,e_{1}^2e_0e_{1}+
(rs)^{-\frac{1}{2}}\times\nonumber\\
(r^{-1}{+}(rs)^{-\frac{1}{2}}{+}s^{-1})\,e_{1}e_0e_{1}^2-(rs)^{-\frac{3}{2}}\,e_0e_{1}^3=0.
\end{eqnarray}
Note that by definition,  the left hand side (LHS) of $(3.21)$ is equal to
\begin{eqnarray*}
LHS&=&\sum\limits_{\substack{i,j, k, m\\ \pi(i)=\pi(j)=\pi(k)=1\\
\pi(m)=0}}\Big[p_i^{\frac{1}{2}}e_i^{\infty}\,p_j^{\frac{1}{2}}e_j^{\infty}\,p_k^{\frac{1}{2}}e_k^{\infty}\,
p_me_m^{\infty}\\
&&-(r^{-1}{+}(rs)^{-\frac{1}{2}}{+}s^{-1})\,p_i^{\frac{1}{2}}e_i^{\infty}\,p_j^{\frac{1}{2}}e_j^{\infty}
\,p_me_m^{\infty}\,p_k^{\frac{1}{2}}e_k^{\infty}\\
&&+(rs)^{-\frac{1}{2}}\,
(r^{-1}{+}(rs)^{-\frac{1}{2}}{+}s^{-1})\,p_i^{\frac{1}{2}}e_i^{\infty}\,p_me_m^{\infty}\,p_j^{\frac{1}{2}}e_j^{\infty}
\,p_k^{\frac{1}{2}}e_k^{\infty}\\
&&-(rs)^{-\frac{3}{2}}\,p_me_m^{\infty}\,p_i^{\frac{1}{2}}e_i^{\infty}\,p_j^{\frac{1}{2}}e_j^{\infty}
\,p_k^{\frac{1}{2}}e_k^{\infty} \Big]
\end{eqnarray*}
Applying Lemma 3.4 through Lemma 3.8, the last relation becomes,
\begin{eqnarray*}
LHS&=&\Big(\sum\limits_{m\gg j>k>i}+\sum\limits_{m=j+1>
k>i}+\sum\limits_{j\gg m\gg k\gg i}+\sum\limits_{j\gg
m=k+1>i}\\
&&+\sum\limits_{j=m+1\gg k>i}+\sum\limits_{j>k\gg m\gg
i}+\sum\limits_{j>k\gg
m=i+1}+\sum\limits_{j>k=m+1>i}\\
&&+\sum\limits_{j>k\gg i\gg m}+\sum\limits_{j>k>i=m+1}\Big)
\{\Big(p_i^{\frac{1}{2}}e_i^{\infty}\,p_j^{\frac{1}{2}}e_j^{\infty}\,p_k^{\frac{1}{2}}e_k^{\infty}\,
p_me_m^{\infty}\\
&&+p_i^{\frac{1}{2}}e_i^{\infty}\,p_k^{\frac{1}{2}}e_k^{\infty}\,p_j^{\frac{1}{2}}e_j^{\infty}\,
p_me_m^{\infty}+p_j^{\frac{1}{2}}e_j^{\infty}\,p_k^{\frac{1}{2}}e_k^{\infty}\,p_i^{\frac{1}{2}}e_i^{\infty}\,
p_me_m^{\infty}\\
&&+p_j^{\frac{1}{2}}e_j^{\infty}\,p_i^{\frac{1}{2}}e_i^{\infty}\,p_k^{\frac{1}{2}}e_k^{\infty}\,
p_me_m^{\infty}+p_k^{\frac{1}{2}}e_k^{\infty}\,p_i^{\frac{1}{2}}e_i^{\infty}\,p_j^{\frac{1}{2}}e_j^{\infty}\,
p_me_m^{\infty}\\
&&+p_k^{\frac{1}{2}}e_k^{\infty}\,p_j^{\frac{1}{2}}e_j^{\infty}\,p_i^{\frac{1}{2}}e_i^{\infty}\,
p_me_m^{\infty} \Big)-(r^{-1}{+}(rs)^{-\frac{1}{2}}{+}s^{-1})\,\\
&&\Big(p_i^{\frac{1}{2}}e_i^{\infty}\,p_j^{\frac{1}{2}}e_j^{\infty}\,
p_me_m^{\infty}\,p_k^{\frac{1}{2}}e_k^{\infty}+p_i^{\frac{1}{2}}e_i^{\infty}\,p_k^{\frac{1}{2}}e_k^{\infty}\,
p_me_m^{\infty}\,p_j^{\frac{1}{2}}e_j^{\infty}\\
&&+p_j^{\frac{1}{2}}e_j^{\infty}\,p_k^{\frac{1}{2}}e_k^{\infty}\,
p_me_m^{\infty}\,p_i^{\frac{1}{2}}e_i^{\infty}+p_j^{\frac{1}{2}}e_j^{\infty}\,p_i^{\frac{1}{2}}e_i^{\infty}\,
p_me_m^{\infty}\,p_k^{\frac{1}{2}}e_k^{\infty}\\
&&+p_k^{\frac{1}{2}}e_k^{\infty}\,p_i^{\frac{1}{2}}e_i^{\infty}\,
p_me_m^{\infty}\,p_j^{\frac{1}{2}}e_j^{\infty}+p_k^{\frac{1}{2}}e_k^{\infty}\,p_j^{\frac{1}{2}}e_j^{\infty}\,
p_me_m^{\infty}\,p_i^{\frac{1}{2}}e_i^{\infty}
\Big)\\
&&+(rs)^{-\frac{1}{2}}\,
(r^{-1}{+}(rs)^{-\frac{1}{2}}{+}s^{-1})\,\Big(p_i^{\frac{1}{2}}e_i^{\infty}\,
p_me_m^{\infty}\,p_j^{\frac{1}{2}}e_j^{\infty}\,p_k^{\frac{1}{2}}e_k^{\infty}\\
&&+p_i^{\frac{1}{2}}e_i^{\infty}\,
p_me_m^{\infty}\,p_k^{\frac{1}{2}}e_k^{\infty}\,p_j^{\frac{1}{2}}e_j^{\infty}+p_j^{\frac{1}{2}}e_j^{\infty}\,
p_me_m^{\infty}\,p_k^{\frac{1}{2}}e_k^{\infty}\,p_i^{\frac{1}{2}}e_i^{\infty}\\
&&+p_j^{\frac{1}{2}}e_j^{\infty}\,p_me_m^{\infty}\,p_i^{\frac{1}{2}}e_i^{\infty}\,
p_k^{\frac{1}{2}}e_k^{\infty}+p_k^{\frac{1}{2}}e_k^{\infty}\,p_me_m^{\infty}\,p_i^{\frac{1}{2}}e_i^{\infty}\,
p_j^{\frac{1}{2}}e_j^{\infty}\\
&&+p_k^{\frac{1}{2}}e_k^{\infty}\,p_me_m^{\infty}\,p_j^{\frac{1}{2}}e_j^{\infty}\,
p_i^{\frac{1}{2}}e_i^{\infty}
\Big)-(rs)^{-\frac{3}{2}}\,\Big(p_me_m^{\infty}\,p_i^{\frac{1}{2}}e_i^{\infty}\,
p_j^{\frac{1}{2}}e_j^{\infty}\,p_k^{\frac{1}{2}}e_k^{\infty}\\
&&+p_me_m^{\infty}\,p_i^{\frac{1}{2}}e_i^{\infty}\,
p_k^{\frac{1}{2}}e_k^{\infty}\,p_j^{\frac{1}{2}}e_j^{\infty}+p_me_m^{\infty}\,p_j^{\frac{1}{2}}e_j^{\infty}\,
p_k^{\frac{1}{2}}e_k^{\infty}\,p_i^{\frac{1}{2}}e_i^{\infty}\\
&&+p_me_m^{\infty}\,p_j^{\frac{1}{2}}e_j^{\infty}\,p_i^{\frac{1}{2}}e_i^{\infty}\,
p_k^{\frac{1}{2}}e_k^{\infty}+p_me_m^{\infty}\,p_k^{\frac{1}{2}}e_k^{\infty}\,p_i^{\frac{1}{2}}e_i^{\infty}\,
p_j^{\frac{1}{2}}e_j^{\infty}\\
&&+p_me_m^{\infty}\,p_k^{\frac{1}{2}}e_k^{\infty}\,p_j^{\frac{1}{2}}e_j^{\infty}\,
p_i^{\frac{1}{2}}e_i^{\infty} \Big) \}=0.
\end{eqnarray*}
Every summand of the last relation can be shown to be 0 as before.

Finally we check that Serre relation involving in $i$ and $i+1$ holds in the Fock space.
For $1\leqslant i \leqslant l-2$, we compute that
\begin{eqnarray*}
&&e_i^2e_{i+1}-(r_i{+}s_i)\,e_ie_{i+1}e_i+(r_is_i)\,e_{i+1}e_i^2\\
&=&\sum\limits_{\substack{j, k, m\\ \pi(j)=i=\pi(k)\\ \pi(m)=i+1}}
\{p_je_j^{\infty}\,p_ke_k^{\infty}\,p_m^{\frac{1}{2}}e_m^{\infty}+p_ke_k^{\infty}\,p_je_j^{\infty}\,p_m^{\frac{1}{2}}e_m^{\infty}\\
&&-(r_i+s_i)(p_je_j^{\infty}\,p_m^{\frac{1}{2}}e_m^{\infty}\,p_ke_k^{\infty}+p_ke_k^{\infty}\,p_m^{\frac{1}{2}}e_m^{\infty}\,p_je_j^{\infty})\\
&&+(r_is_i)(p_m^{\frac{1}{2}}e_m^{\infty}\,p_je_j^{\infty}\,p_ke_k^{\infty}+p_m^{\frac{1}{2}}e_m^{\infty}\,p_ke_k^{\infty}\,p_je_j^{\infty})\}
\end{eqnarray*}
\begin{eqnarray*}
&=&\Big(\sum\limits_{m\gg j>k}+\sum\limits_{m=j+1\gg
k}+\sum\limits_{j\gg m\gg k}+\sum\limits_{j\gg
m=k+1}+\sum\limits_{j=m+1>k}\\
&&+\sum\limits_{j>k\gg
m}+\sum\limits_{j>k=m+1}\Big)\{p_je_j^{\infty}\,p_ke_k^{\infty}\,p_m^{\frac{1}{2}}e_m^{\infty}+p_ke_k^{\infty}\,p_je_j^{\infty}\,p_m^{\frac{1}{2}}e_m^{\infty}\\
&&-(r_i+s_i)(p_je_j^{\infty}\,p_m^{\frac{1}{2}}e_m^{\infty}\,p_ke_k^{\infty}+p_ke_k^{\infty}\,p_m^{\frac{1}{2}}e_m^{\infty}\,p_je_j^{\infty})\\
&&+(r_is_i)(p_m^{\frac{1}{2}}e_m^{\infty}\,p_je_j^{\infty}\,p_ke_k^{\infty}+p_m^{\frac{1}{2}}e_m^{\infty}\,p_ke_k^{\infty}\,p_je_j^{\infty})\}=0.
\end{eqnarray*}
Therefore we have finished the proof of Theroem 3.2.
\end{proof}

\vskip30pt \centerline{\bf ACKNOWLEDGMENT}

\bigskip

Jing thanks the support of
Simons Foundation grant 198129, NSFC grant 11271138 and NSF grants
1014554 and 1137837. H. Zhang would
like to thank the support of NSFC grant 11371238.

\bigskip

\bibliographystyle{amsalpha}

\end{document}